\documentclass
{amsart}
\usepackage{graphicx}
\newtheorem{thm}{Theorem}[section]
\theoremstyle{definition}
\newtheorem{cor}[thm]{Corollary}
\newtheorem{prop}[thm]{Proposition}
\newtheorem{defn}[thm]{Definition}

\newtheorem{lem}[thm]{Lemma}

\newtheorem{rem}[thm]{Remark}
\newtheorem{ex}[thm]{Example}

\numberwithin{equation}{section}
 \usepackage[paperwidth=165mm, paperheight=235mm, twoside, hmargin={25mm,20mm}, vmargin={20mm,20mm} ]{geometry}
\begin{document}
\title[Nil-prime ideals of a commutative ring]
{Nil-prime ideals of a commutative ring}

\author[Faranak Farshadifar]%
{Faranak Farshadifar*}

\newcommand{\acr}{\newline\indent}
\address{\llap{*\,} Department of Mathematics Education, Farhangian University, P.O. Box 14665-889, Tehran, Iran.}
\email{f.farshadifar@cfu.ac.ir}


% ----------------------------------------------------------------
\begin{abstract}
Let $R$ be a commutative ring with identity and $Nil(R)$ be the set of all nilpotent elements of $R$.
The aim of this paper is to introduce and study the notion of nil-prime ideals as a generalization of prime ideals.
We say that a proper ideal $P$ of $R$ is a \textit{nil-prime ideal} if there exists $x \in Nil(R)$ such that whenever $ab \in P$,  then $a\in P$ or $b \in P$ or $a+x \in P$ or $b+x \in P$ for each $a, b \in R$. Also, we introduce nil versions of some algebraic concepts in ring theory such as nil-maximal ideal, nil-minimal ideal, nil-principal ideal and investigate some nil-version of a well-known results about them.
\end{abstract}
\maketitle
% ----------------------------------------------------------------
\section{Introduction}
\noindent
Throughout this paper, $R$ will denote a commutative ring with
identity and $\Bbb Z$ will denote the ring of integers. Also, $Nil(R)$ will denote the set of all nilpotent elements of $R$.

A proper ideal $P$ of $R$ is said to be a \textit{prime ideal} if $ab \in P$ for some
$a, b \in R$, then either $a \in P$ or $b \in P $ \cite{MR3525784}.
Theory of prime ideals is an important tool in classical algebraic geometry.
In development of algebraic geometry, some generalizations
for the concept of prime ideals has arisen.
For example in \cite{MR4127389}, the authors  introduced and studied the notion of $S$-prime ideals in a commutative ring.
Let $S \subseteq R$ be a multiplicative set and $P$ an ideal
of $R$ disjoint with $S$. Then $P$ is said to be \textit{$S$-prime} if there exists an $s \in S$ such that for all
$a, b \in R$ with $ab \in P$, we have $sa \in P$ or $sb \in P$.
Motivated by $S$-prime ideals, it is natural to ask what is a nil-version of prime ideals? In this regard,
in \cite{SU18}, the authors introduced and investigated the notion of $\mathfrak{N}$-prime ideals as a generalization of prime ideals.
A proper ideal $P$ of $R$ is said to be a \textit{$\mathfrak{N}$-prime ideal} if $ab \in P$, for each
$a, b \in R$, then either $a \in P + Nil(R)$ or $b \in P + Nil(R)$ \cite{SU18}.

The aim of this paper is to introduce the notions of nil-prime ideals as a generalization of prime ideals and investigate some nil-versions of  well-known results about prime ideals.
We say that a proper ideal $P$ of $R$ is a \textit{nil-prime ideal} if there exists $x \in Nil(R)$ such that whenever $ab \in P$,  then $a \in P$ or $b \in P$ or $a+x \in P$ or $b+x \in P$ for each $a, b \in R$ (Definition \ref{1.1}).
It is shown that the class of nil-prime ideals is located properly between the class of prime ideas and the class of $\mathfrak{N}$-prime ideals. Also, we introduce nil versions of some algebraic concepts in ring theory such as nil-maximal ideal, nil-minimal ideal, nil-principal ideal, $\mathfrak{N}(R)$-integral domain, and $\mathfrak{N}(R)$-PID  that are needed in the sequel.

%%%%%%%%%%%%%%%%%%%%%%%%%%%%%%%
%%%%%%%%%%%%%%%%%%%%%%%%%%%%%%%%%%%%%%
%%%%%%%%%%%%%%%%%%%%%%%%%%%%%%%%%%%%%%%%%%%%%
%%%%%%%%%%%%%%%%%%%%%%%%%%%%%%%%%%%%%%%%%%%%%%%%
\section{Nil-prime ideals}
\begin{defn}\label{1.1}
We say that a proper ideal $P$ of $R$ is a \textit{nil-prime ideal} if there exists $x \in Nil(R)$ such that whenever $ab \in P$,  then $a \in P$ or $b \in P$ or $a+x \in P$ or $b+x \in P$ for each $a, b \in R$. In this case, $P$ is said to be \textit{nil-prime ideal with respect to} $x$.
\end{defn}

In the following examples and remarks, we can see that the  class of nil-prime ideals is located properly between the class of prime ideals and the class of $\mathfrak{N}$-prime ideals.

Let $n$ be a positive integer. Consider the ring $\Bbb Z_n$ of integers modulo $n$. We know that $\Bbb Z_n$ is a principal ideal ring and each of these ideals is generated by $\bar{m} \in \Bbb Z_n$, where $m$ is a factor of $n$. In this paper, we denote this ideal by $\langle m\rangle$.
\begin{rem}\label{1.2}
Clearly every prime ideal of $R$ is a nil-prime ideal of $R$. But the Example \ref{1.3} shows that the converse is not true in general. If $P$ is a nil-prime ideal of $R$ such that $Nil(R) \subseteq P$, then $P$ is a prime ideal of $R$. Therefore, if $Nil(R)=0$ (i.e., $R$ is reduced), then the notions of prime ideals and nil-prime ideals are equal. For example, if $n$ is square-free (i.e., $n$ has not a square factor), then $Nil(\Bbb Z_n)=0$. Also, $Nil(\Bbb Z)=0$ and $Nil(F[x])=0$,  where $F$ is a field.
\end{rem}

\begin{ex}\label{1.3}
Consider the ideal $\langle 0\rangle$ of the ring $\Bbb Z_8$ and $4  \in Nil(\Bbb Z_8)=\{2^k: k\ is\ a\ positive\ integer\}$. Since $(2)(4)=8\in \langle 0\rangle$, $2, 4 \not\in\langle 0\rangle$, we have $\langle 0\rangle$ is not a prime ideal of $\Bbb Z_8$. But $4+4=8 \in  \langle 0\rangle$. This in turn implies that $\langle 0\rangle$ is nil-prime ideal.
\end{ex}

\begin{rem}\label{1.6}
Clearly every nil-prime ideal of $R$ is a $\mathfrak{N}$-prime ideal of $R$. But the Example \ref{1.7} shows that the converse is not true in general.
\end{rem}

\begin{ex}\label{1.7}
Consider the ideal $P=\langle 16\rangle$ of the ring $R=\Bbb Z_{32}$. Then by \cite[Example 3.1]{SU18}, $P$ is a $\mathfrak{N}$-prime ideal of $R$.
But $P$ is not a nil-prime ideal of $R$. Because $(2)(8) \in P$ and $(4)(4) \in P$ but $2,4,8 \not \in P$ and there is not $x \in Nil(R)$ such that $4+x\in P$ and also $2+x\in P$ or $8+x \in P$.
Although, $2+14 \in P$, $8+8 \in P$, and $4+12 \in P$ for $14, 8, 12\in Nil(R)$.
\end{ex}

\begin{prop}\label{001.16}
Let $P$ be a nil-prime ideal of $R$. Then $\sqrt{P}$ is a prime ideal of $R$.
\end{prop}
\begin{proof}
Let $ab \in \sqrt{P}$. Then $a^nb^n\in P$ for some positive integer $n$. Thus there exists $x \in  Nil(R)$ such that  $a^n \in P$ or $b^n \in P$ or $a^n+x \in P$ or $b^n+z \in P$. Hence,
$a^n \in P+Rx\subseteq \sqrt{P}$ or $b^n \in P+Rx\subseteq\sqrt{P}$. Therefore, $a\in \sqrt{P}$ or $b \in \sqrt{P}$, as needed.
\end{proof}

\begin{thm}\label{001.196}
Let $P$ be a nil-prime ideal of $R$ with respect to $x$. Then we have the following.
\begin{itemize}
\item [(a)] $2x\in P$.
\item [(b)] For each $a, b \in R\setminus P$ with $ab \in P$, we have $2a\in P$ or $2b\in P$.
\end{itemize}
\end{thm}
\begin{proof}
(a) Since $x \in  Nil(R)$, there exists $n\in \Bbb N$ such that $x^n=0$. If $x \in P$, we are done. So suppose that $x \not \in P$. Then as $(x)(x^{n-1})=x^n=0 \in P$ we have $2x \in P$ or $x^{n-1} \in P$ or $x^{n-1}+x \in P$. If $2x \in P$, then we are done. If $x^{n-1}+x \in P$, then $x^2=0+x^2=x(x^{n-1}+x) \in P$. This implies that $2x\in P$. If $x^{n-1} \in P$, then by continueing in this way, we get that $2x\in P$.

(b) Let $a, b \in R\setminus P$ with $ab \in P$. Then $a+x \in P$ or $b +x\in P$. Without loss of generality, assume that $a+x \in P$. Then
  $a-x+(x+x)=-x+x+a+x \in P$. It follows that $a-x\in P$ since $x+x\in P$ by part (a). Therefore, $2a=a+x+a-x\in P$.
\end{proof}

\begin{defn}\label{1.4}
\begin{itemize}
\item [(a)] We say that a proper ideal $M$ of $R$ is a \textit{nil-maximal ideal} if there exists $x \in Nil(R)$ and whenever $M\subseteq I\subseteq R$,  then $I= M$ or $I= M+Rx$ or $I=R$.
\item [(b)] We say that a proper ideal $M$ of $R$ is a \textit{$\mathfrak{N}$-maximal ideal} if whenever $M\subseteq I\subseteq R$,  then $I+Nil(R)=R$ or  $I\subseteq M+Nil(R)$.
\end{itemize}
\end{defn}
Clearly, every nil-maximal ideal of $R$ is a $\mathfrak{N}$-maximal ideal of $R$. But the Example \ref{11.192}, shows that the converse is not true in general.

\begin{thm}\label{1.5}
Let $P$ be a nil-maximal ideal of $R$. Then $P$ is a $\mathfrak{N}$-prime ideal of $R$.
\end{thm}
\begin{proof}
Suppose that $ab \in P$. Since $P\subseteq Ra+P\subseteq R$ and $P$ is a nil-maximal ideal of $R$, we have there exists $x \in Nil(R)$ such that $Ra+P=P$ or  $Ra+P=P+Rx$ or  $Ra+P=R$.
Also, for $x \in Nil(R)$, we have $Rb+P=P$ or  $Rb+P=P+Rx$ or  $Rb+P=R$. If $Ra+P=P$ or  $Rb+P=P$, we are done.
If $Rb+P=R$ and $Ra+P=R$, then
$$
R=P+Ra=P+(P+Rb)a=P+Pa+Rab=P.
$$
This is a contradiction because $P$ is proper. So, we can suppose that $Ra+P=P+Rx$.  Thus $a=a+0\in Ra+P\subseteq P+Rx$. It follows that  $a\in  P+Nil(R)$, as needed.
\end{proof}

The following example shows that the converse of Theorem \ref{1.5} is not true in general.
\begin{ex}\label{1.192} Let $R=\Bbb Z_8[X,Y]$ and $P=\langle \bar{4}XY\rangle$.
Then $P$ is a $\mathfrak{N}$-prime ideal of $R$ (see \cite[Example 2.2 (ii)]{SU18}). But since $P \subset \langle X\rangle \subset R$ and  $\langle X\rangle \not = P+Rt\subseteq P+Nil(R)=Nil(R)=\bar{2}\Bbb Z_8[X,Y]$ for each $t \in Nil(R)$ we have $P$ is not a nil-maximal ideal of $R$.
\end{ex}

\begin{thm}\label{1.12}
Let $f:  R \rightarrow S$ be an epimorphism and $P$ be a nil-prime
ideal of $R$ such that $Ker(f) \subseteq P$. Then $f(P)$ is a nil-prime ideal of $S$.
\end{thm}
\begin{proof}
Clearly, $f(P)\not=S$. Assume that  $ab \in  f(P) $ for some $a,b \in S$. As $f$ is an epimorphism, we have $a=f(x)$ and $b=f(y)$ for some $x, y \in R$. Thus we have $xy \in f^{-1}(f(P))=P$. As $P$ is a nil-prime ideal of $R$, there exists $t \in Nil(R)$
such that $x \in P$ or $y \in P$ or $x+t \in P$ or $y+t \in P$. Thus $f(x)\in f(P)$ or $f(y)\in f(P)$ or $f(x)+f(t)\in f(P)$ or $f(y)+f(t)\in f(P)$.
Now since $f(t) \in Nil(S)$, $f(P)$ is a nil-prime ideal of $S$.
\end{proof}

The following corollary is now evident.
\begin{cor}\label{1.13}
If $P$ is a nil-prime ideal of $R$ that contains an ideal $I$, then $P/I$ is
a nil-prime ideal of $R/I$.
\end{cor}

\begin{prop}\label{1.14}
Let $P$ be a proper ideal of $R$. If $\langle P, X\rangle$ is a nil-prime ideal of $R[X]$, then $P$ is a nil-prime ideal of $R$.
\end{prop}
\begin{proof}
Consider the homomorphism $\phi: R[X] \rightarrow R$ defined by
 $\phi (f(X)) = f(0)$. Clearly, $ Ker(\phi) = \langle X\rangle \subseteq  \langle P, X\rangle$ and $\phi$ is an epimorphism. As $\langle P, X\rangle$ is a nil-prime ideal of $R[X]$, we have $\phi ( \langle P, X\rangle)=P$ is a nil-prime ideal of $R$ by Theorem \ref{1.12}.
\end{proof}

Let $R_1$, $R_2$ be two commutative rings. Then $R = R_1\times R_2$ becomes
a commutative ring under componentwise addition and multiplication. In addition,
every ideal $I$ of $R$ has the form $I_1 \times I_2$, where $I_i$ is an ideal of $R_i$ for $i = 1, 2$.

\begin{lem}\label{1.10}
Let $R = R_1\times R_2$ and $P = P_1\times P_2$,  where $P_i$ is an ideal of $R_i$ for
$i = 1, 2$. Then the followings are equivalent:
\begin{itemize}
\item [(a)] $P$ is a nil-prime ideal of $R$;
\item [(b)] $P_1$ is a nil-prime ideal of $R_1$ and $P_2 = R_2$ or $P_1 = R_1$ and $P_2$ is a nil-prime
ideal of $R_2$.
\end{itemize}
\end{lem}
\begin{proof}
$(a)\Rightarrow (b)$.
By Proposition \ref{001.16},
 $\sqrt{P}=\sqrt{P_1}\times \sqrt{P_2}$ is a prime ideal. Therefore, we have either $\sqrt{P_1}= R_1$ or $\sqrt{P_2 }= R_2$ by \cite[Theorem 6]{MR2454978}. This implies that  $P_1= R_1$ or $P_2= R_2$. So we can assume that $P_1 = R_1$. Now we prove that $P_2$ is a
nil-prime ideal of $R_2$. So suppose that  $a_2b_2\in P_2$ for some $a_2, b_2 \in R_2$. Then there exists $(x_1,x_2)\in Nil(R_1 \times  R_2)$ such that $(0, a_2)(0, b_2) = (0, a_2b_2) \in  P$ implies that $(0, a_2) \in  P$ or $(0, b_2)\in P$. or 
$(0, a_2) +(x_1, x_2)\in  P$ or $(0, b_2)+ (x_1, x_2)\in P$. Thus  $a_2 \in P_2$ or $b_2 \in P_2$ or $a_2+x_2 \in P_2$ or $b_2+x_2 \in P_2$, as needed.

$(b)\Rightarrow (a)$.
Assume that $P = P_1\times R_2$,  where $P_1$ is a nil-prime ideal of $R_1$. We show that $P$ is a nil-prime ideal of $R$. So let
$(a_1,a_2)(b_1,b_2) \in P_1\times R_2$. Then $a_1b_1 \in P_1$. Hence there exists $x_1 \in Nil(R_1)$ such that $a_1 \in P_1$ or $b_1 \in P_1$ or 
$a_1+x_1 \in P_1$ or $b_1+x_1 \in P_1$. This implies that $(a_1,a_2) \in P_1\times R_2$ or $(b_1,b_2) \in P_1\times R_2$ or  $(a_1,a_2)+(x_1,0) \in P_1\times R_2$ or $(b_1,b_2)+(x_1,0) \in P_1\times R_2$ for $(x_1, 0) \in Nil(R_1\times R_2)$. Thus $P$ is a nil-prime ideal of $R$.
\end{proof}

\begin{thm}\label{1.11}
Let $R=R_1\times R_2\times...\times R_n$, where $n \geq 2$, and
$P = P_1\times P_2\times...\times P_n$,
where $P_i$ is an ideal of $R_i$, $1 \leq i \leq n$. Then the followings are equivalent:
\begin{itemize}
\item [(a)] $P$ is a nil-prime ideal of $R$;
\item [(b)] $P_j$ is a nil-prime ideal of $R_j$ for some $j \in \{1,2,...,n\}$ and $P_i = R_i$ for
each $i \not= j$.
\end{itemize}
\end{thm}
\begin{proof}
We use induction on $n$. By Lemma \ref{1.10}, the claim is true if $n = 2$. So, suppose that the claim is true for each $k  \leq n-1$ and let $k = n$. Put $Q= P_1\times P_2\times...\times P_{n-1}$,
and $\acute{R} = R_1\times R_2\times...\times R_{n-1}$, by Lemma \ref{1.10}, $P = Q\times P_n$ is a nil-prime ideal of
$R = \acute{R} \times R_n$ if and only if $Q$ is a nil-prime ideal of $\acute{R}$ and $P_n = R_n$ or $Q=\acute{R}$ and
$P_n$ is a nil-prime ideal of $R_n$. Now the rest follows from induction hypothesis.
\end{proof}

Let $M$ be an $R$-module and $R \oplus M = \{(a,m) : a \in R, m \in  M\}$. Then $R \oplus M$, \textit{idealization of $M$}, is a commutative ring with componentwise
addition and the multiplication: $(a,m_1)(b,m_2) = (ab,am_2 + bm_1)$ \cite{H988}.
If $P$ is an ideal of $R$ and $N$ is a submodule of $M$, then $P \oplus N$ is an ideal of $R\oplus M$ if
and only if $PM \subseteq  N$. Then $P \oplus N$ is called a \textit{homogeneous ideal}. In \cite{AW09}, it was
shown that $Nil(R \oplus M) = Nil(R) \oplus M$ and then all prime ideals $P$ of $R \oplus M$ are of
the form $P = P_1 \oplus M$, where $P_1$ is a prime ideal of $R$.

\begin{prop}\label{001.199996}
Let $M$ be an $R$-module, $P$ an ideal of $R$, and let $N$ be
a proper submodule of $M$ such that $PM \subseteq N$. If $P \oplus N$ is a nil-prime ideal of $R \oplus M$
with respect to $(x,m)$, then we have the following.
\begin{itemize}
\item [(a)] $m\in M\setminus N$, $x \in P$, and $2m\in N$.
\item [(b)] For each $m_1 \in M\setminus N$, we have $2m_1\in N$.
\end{itemize}
\end{prop}
\begin{proof}
(a) Let $P \oplus N$ be a nil-prime ideal of $R \oplus M$ with respect to $(x,m)$ and let $m_1 \in M\setminus N$. Then $(0,m_1)(0,m_1)=(0,0) \in P \oplus N$ implies that 
$(0,m_1) \in P \oplus N$ or $(0,m_1)+(x,m) \in P \oplus N$.  Since $m_1 \not \in N$, we have $(0,m_1)+(x,m) \in P \oplus N$. Thus $x \in P$ and $m_1+m \in N$. As $m_1 \not \in N$, we get that $m\in M\setminus N$.
As, $(0,m)(0,m)=(0,0) \in P \oplus N$ implies that 
$(0,m) \in P \oplus N$ or $(0,m)+(x,m) \in P \oplus N$.  Since $m \not \in N$, we have $(0,m)+(x,m) \in P \oplus N$. Thus $2m \in N$.

(b) Let $m_1 \in M\setminus N$. Then $(0,m_1)(0,m_1)=(0,0) \in P \oplus N$ implies that 
$(0,m_1)+(x,m) \in P \oplus N$. Then $m_1+m \in N$. So,
$m_1+2m-m\in N$. By part (a), $2m\in N$. Thus $m_1-m\in N$. Therefore, $2m_1=m_1+m+m_1-m\in N$.
\end{proof}

\begin{thm}\label{1.15}
Let $M$ be an $R$-module and $P$ be an ideal of $R$. Then we have the following.
\begin{itemize}
\item [(a)] If $N$ is a submodule of $M$ such that $PM \subseteq N$ and $P \oplus N$ is a nil-prime ideal of $R \oplus M$, then $P$ is a nil-prime ideal of $R$.
\item [(b)] If $P$ is a nil-prime ideal of $R$, then $P\oplus M$ is a nil-prime ideal of $R \oplus M$.
\end{itemize}
\end{thm}
\begin{proof}
(a) Let $N$ be a submodule of $M$ such that $PM \subseteq N$ and $P \oplus N$ be a nil-prime ideal of $R \oplus M$. Assume that $ab \in P$ for $a, b \in R$. Then
$(a, 0)(b, 0) = (ab, 0) \in P \oplus N$. By assumption, there exists $(x,m) \in Nil(R \oplus M) = Nil(R) \oplus M$ such that $(a, 0) \in P \oplus N$ or $(b, 0) \in P \oplus N$ or $(a, 0)+(x,m) \in P \oplus N$ or $(b, 0)+(x,m) \in P \oplus N$.
Therefore, $x \in  Nil(R)$ and $a \in P$ or $b \in P$ or $a+x \in P$ or $b+x \in P$ as needed.

(b) Let $P$ be a nil-prime ideal of $R$ and $(a, m_1)(b, m_2) \in P \oplus M$. Then $ab \in P$ and by assumption, there exists $x \in  Nil(R)$ such that
$a \in P$ or $b \in P$ or $a+x \in P$ or $b+x \in P$. Hence  $(a, m_1) \in P \oplus N$ or $(b, m_2) \in P \oplus N$ or $(a, m_1)+(x,0) \in P \oplus N$ or $(b, m_2)+(x,0) \in P \oplus N$.
Since $(x,0) \in Nil(R \oplus M)$, we have $P \oplus M$ is a nil-prime ideal of $R \oplus M$.
\end{proof}

\begin{defn}\label{91.193}
We say that two ideals $I$ and $J$ of $R$ are \textit{nil-distinct} if $J \not \subseteq I+Rz$ and $I\not\subseteq J+Rz$ for each $z \in Nil(R)$.
\end{defn}

\begin{lem}\label{27262.1}
Let $P$ be a $\mathfrak{N}$-prime ideal of $R$ and $I_1, I_2,\ldots, I_n$ be ideals of $R$ such that
$I_1 I_2\ldots I_n\subseteq P$. Then $I_i\subseteq P+Nil(R)$ for some i ($1\leq i\leq n$).
\end{lem}
\begin{proof}
By \cite[Proposition 2.1]{SU18}, $P+Nil(R)$ is a prime ideal of $R$. Thus the result follows from the fact that $I_1 I_2\ldots I_n\subseteq P\subseteq P+Nil(R)$.
\end{proof}

It is well known that in Artinian ring, every prime ideal is a maximal ideal and Artinian ring has only a finite number of maximal ideals \cite{MR3525784}.
The following theorem is a nil-versions of these facts.
\begin{thm}\label{1.176}
Let $R$ be an Artinian ring. Then we have the following.
\begin{itemize}
\item [(a)] If $P$ is a $\mathfrak{N}$-prime ideal of $R$, then $P$ is a $\mathfrak{N}$-maximal ideal of $R$.
\item [(b)] $R$ has only a finite number of nil-maximal ideals which are nil-distinct.
\end{itemize}
\end{thm}
\begin{proof}
(a) Let $P$ be a $\mathfrak{N}$-prime ideal of $R$ and $P\subseteq I\subseteq R$ for some ideal $I$ of $R$.
Assume that $x \in I$. Then as $R$ is an Artinian ring, for the following descending chain
$$
Rx \supseteq Rx^2 \supseteq \cdots \supseteq Rx^t \supseteq\cdots
$$
we have $Rx^n=Rx^{n+1}$ for some positive integer $n$. Thus $(1-xr)x^n=0 \in P$. Now since $P$ is $\mathfrak{N}$-prime,
we have  $x^n \in P+Nil(R)$ or $1-rx \in P+Nil(R)$. By \cite[Proposition 2.1]{SU18}, $P+Nil(R)$ is a prime ideal of $R$. Therefore, 
 $x \in P+Nil(R)$ or $I+Nil(R)=R$.  Thus $I\subseteq P+Nil(R)$ or $I+Nil(R)=R$, as needed.
 
(b) Consider the set of all finite intersections $M_1\cap \cdots \cap M_t$, where the $M_i$ are nil-maximal ideals of $R$ which are
nil-distinct.
Since $R$ is Artinian, this set has a minimal element, say $M_1\cap \cdots \cap M_n$. Hence for any nil-maximal ideal $M$ or $R$ which is nil-distinct with $M_i$, we have $M\cap M_1\cap \cdots \cap M_n=M_1\cap \cdots \cap M_n$.
Thus $M_1\cdots  M_n\subseteq M_1\cap \cdots \cap M_n \subseteq M$. By Theorem \ref{1.5}, $M$ is a $\mathfrak{N}$-prime ideal of $R$.
So, by Lemma \ref{27262.1}, $M_i\subseteq M +Nil(R)$ for some $i$.
Now as $M_i$ is a nil-maximal ideal of $R$, we have  $M_i\subseteq M +Nil(R) \subseteq R$,
implies that $M_i=M +Nil(R)$ or $R=M +Nil(R)$ or $M +Nil(R)=M_i+Rx$ for some $x \in Nil(R)$.
If $M_i=M +Nil(R)$, then $M\subseteq M +Nil(R)=M_i\subseteq M_i+Rx$ for each  $x \in Nil(R)$, 
which is a contradiction since $M$ and $M_i$ are nil-distinct.
If $R=M +Nil(R)$, then $1=a+y$ for some $y \in Nil(R)$. It follows that $R=M+Ry$ and so $R=M$, which is a contradiction. If $M +Nil(R)=M_i+Rx$, then for each $a \in M$ we have $a+x=a_i+rx$ for some $a_i \in M_i$ and $r \in R$. This implies that $a=a_i+(1-r)x\subseteq M_i+Rx$ and so $M\subseteq  M_i+Rx$. Which is a desired contradiction because $M$ and $M_i$ are nil-distinct.
\end{proof}

\begin{ex}\label{11.192}
Consider the ideal $P=\langle 16\rangle$ of the ring $R=\Bbb Z_{32}$. Then by Example \ref{1.7}, $P$ is a $\mathfrak{N}$-prime ideal of $R$. As $R$ is an Artinian ring, we have $P$ is a 
$\mathfrak{N}$-maximal ideal of $R$ by Theorem \ref{1.176} (a).
But $P$ is not a nil-maximal ideal of $R$. Because 
$P\subseteq \langle 8\rangle \subseteq R$ and $P\subseteq \langle 4\rangle \subseteq R$. But one can see that there is not $x \in Nil(R)$ such that 
$\langle 4\rangle=P+Rx$ and $\langle 8\rangle=P+Rx$.
\end{ex}

\begin{defn}\label{1.17}
We say that a non-zero ideal $I$ of $R$ is a \textit{nil-minimal ideal} if there exists $x \in Nil(R)$ and whenever $0\subseteq J \subseteq I$ for some ideal $J$ of $R$,  then $I=J+Rx$ or $J=Rx$.
\end{defn}

\begin{defn}\label{222.1}
\begin{itemize}
\item [(a)] We say that an ideal $I$ of $R$ is a \textit{nil-principal ideal} if there exist $r \in R$ and $x \in Nil(R)$  such that $I= Rr+Rx$.
\item [(b)] We say that an ideal $I$ of $R$ is a $\mathfrak{N}$-\textit{principal ideal} if there exists $r \in R$ such that $I\subseteq Rr+Nil(R)$.
\end{itemize}
\end{defn}

Clearly, every nil-principal ideal is a $\mathfrak{N}$-principal ideal.
\begin{thm}\label{1.16}
Let $I$ be a nil-minimal ideal of $R$  such that $I\not \subseteq Nil(R)$. Then we have the followings.
\begin{itemize}
\item [(a)] $I$ is a nil-principal ideal of $R$.
\item [(b)] $(Nil(R):_RI)$ is a maximal ideal of $R$.
\end{itemize}
\end{thm}
\begin{proof}
(a) Let $a \in I\setminus Nil(R)$. Then $0\subseteq Ra\subseteq I$ implies that there exists $x \in Nil(R)$ such that $I=Ra+Rx$ or $Ra=Rx$. Since $a \in I\setminus Nil(R)$, we have $I=Ra+Rx$.

(b) Let $I\not \subseteq Nil(R)$. Then $(Nil(R):_RI)$ is a proper ideal of $R$. Suppose that $(Nil(R):_RI)\subseteq J\subseteq R$. Then $0\subseteq IJ\subseteq I$ implies that $I=IJ+Ry$  or $IJ=Ry$ for $y \in Nil(R)$ because $I$ is nil-minimal ideal.
 If $IJ=Ry$, then since $Ry\subseteq Nil(R)$, we have $J\subseteq (Nil(R):_RI)$. Thus $J=(Nil(R):_RI)$ and we are done.  So assume that $IJ\not=Ry$. By part (a),  $I=Ra+Rx$ for some $a \in I\setminus Nil(R)$ and $x \in Nil(R)$. Therefore, $a \in I=aJ+xJ+yR$. Thus
$a=aj_1+xj_2+sy$ for some $j_1, j_2 \in J$ and $s \in R$. It follows that $(1-j_1)a \in Nil(R)$. Hence, $(1-j_1)I=R(1-j_1)a+R(1-j_1)x\subseteq Nil(R)$.
Thus $1-j_1 \in (Nil(R):_RI)\subseteq J$ and so $1\in J$. Hence $J=R$, as needed.
\end{proof}

\begin{defn}\label{2262.1}
We say that a commutative ring $R$ is a $\mathfrak{N}$-\textit{integral domain} if the zero ideal of $R$ is a $\mathfrak{N}$-prime ideal of $R$.
\end{defn}

\begin{defn}\label{2262.1}
We say that a nil integral domain $R$ is a $\mathfrak{N}$-\textit{PID} if every ideal of $R$ is a $\mathfrak{N}$-principal ideal of $R$.
\end{defn}

Let $R$ be an integral domain. It is well known that $R$ is a PID if and only if each prime
ideal of $R$ is principal. The following theorem is a $\mathfrak{N}$-version of this fact.

\begin{thm}\label{1.16}
Let $R$ be a $\mathfrak{N}$-integral domain. Then $R$ is a $\mathfrak{N}$-PID if and only if each $\mathfrak{N}$-prime
ideal of $R$ is $\mathfrak{N}$-principal.
\end{thm}
\begin{proof}
The direct implication follows directly from the definition.

For the reverse implication, suppose that every nil-prime ideal of $R$ is $\mathfrak{N}$-principal. Assume,
by way of contradiction, that $R$ is not a $\mathfrak{N}$-PID, and so that there is an ideal of $R$ that is
not $\mathfrak{N}$-principal. Then the set $\Omega$ consisting of all non-$\mathfrak{N}$-principal ideals of $R$ is a non-empty
partially ordered set. Suppose that $\{I_{\gamma} : \gamma \in  \Gamma \}$ is a chain in $\Omega$. It is not hard to
verify that $I :=\cup_{\gamma \in  \Gamma}I_{\gamma}$
is a non-$\mathfrak{N}$-principal ideal of $R$ and, therefore, an upper bound
for the given chain. Then $\Omega$ contains a maximal element $\mathfrak{M}$ by Zorn’s lemma.
Since $\mathfrak{M}$ is not $\mathfrak{N}$-principal, it cannot be nil-prime. Thus, there exist $x_1, x_2 \in  R$ such
that $x_1x_2 \in \mathfrak{M}$ and $x_1+y\not \in \mathfrak{M}$, $x_2+y\not \in \mathfrak{M}$ for each $y \in Nil(R)$. Since the ideals $I_1:= \mathfrak{M} + Rx_1$ and $I_2 := \mathfrak{M} +Rx_2$ properly contain $\mathfrak{M}$,
the maximality of $\mathfrak{M}$ in $\Omega$ guarantees the existence of $a \in  R$ such that $I_1 \subseteq Ra+Nil(R)$. Define
$$
K := (\mathfrak{M}+Nil(R):_R I_1) = \{r \in R : rI_1 \subseteq \mathfrak{M}+Nil(R)\}.
$$
As $I_1I_2=\mathfrak{M}^2+x_2\mathfrak{M}+x_1\mathfrak{M}+Rx_1x_2\subseteq \mathfrak{M}\subseteq \mathfrak{M}+Nil(R)$, we have $I_2\subseteq K$.
This implies that $\mathfrak{M} \subset K$.
So $K$ must be $\mathfrak{N}$-principal, and we can take $b \in R$ such that $K \subseteq bR+Nil(R)$.
Let  $c\in  \mathfrak{M}$. Since $\mathfrak{M} \subseteq I_1$, we can write $c = ra+z$ for some
$r \in R$ and $z \in Nil(R)$. If $t \in rI_1\subseteq Rar+rNil(R)$, then for some $s \in R$ and $z_1 \in Nil(R)$ we have
$$
 t=sar+rz_1=sar+sz-sz+rz_1=sc-sz+rz_1\in \mathfrak{M}+Nil(R).
$$
 Thus $rI_1\subseteq \mathfrak{M}+Nil(R)$.
 It follows that $r\in  K$. Hence $r=br_1+w$ for some $r_1 \in R$ and $w\in Nil(R)$. Therefore,
$$
c= ra+z=(br_1+w)a+z=bar_1+wa+z\in Rab+Nil(R).
$$
 So we have $\mathfrak{M}\subseteq Rab+Nil(R)$, which
contradicting the fact that $\mathfrak{M}$ belongs to $\Omega$. Therefore, $R$ is a $\mathfrak{N}$-PID.
\end{proof}
%%%%%%%%%%%%%%%%%%%%%%%%%%%%%%%
%%%%%%%%%%%%%%%%%%%%%%%%%%%%%%%%%%%%%%
%%%%%%%%%%%%%%%%%%%%%%%%%%%%%%%%%%%%%%%%%%%%%
%%%%%%%%%%%%%%%%%%%%%%%%%%%%%%%%%%%%%%%%%%%%%%%%
{\bf Acknowledgement.}
The author would like to thank Professor Hani A. Khashan for his helpful suggestions and useful comments.


\begin{thebibliography}{99}
\bibitem{MR2454978}
D. D. Anderson and John. Kintzinger, \emph{Ideals in direct products of commutative rings}, Bull. Aust. Math. Soc., \textbf{77} (3) (2008),  477-483.

\bibitem{AW09}
 D.D. Anderson and M.Winders, \emph{Idealization of a module}, Journal of Commutative
Algebra, \textbf{1}(1) (2009), 3-56.

\bibitem{MR3525784}
M. F. Atiyah and I. G. Macdonald, \emph{Introduction to commutative algebra}, Westview Press, Boulder, CO, (2016).

\bibitem{MR4127389}
H. Ahmed and M. Achraf, \emph{$S$-prime ideals of a commutative ring}, Beitr. Algebra Geom.,  \textbf{61} (3) (2020), 533-542.

\bibitem{H988}
J.A. Huckaba,  \emph{Commutative rings with zero divisors}, Taylor  Francis, 1988.

\bibitem{SU18}
E.S. Sevim and S. Koc, \emph{On $\mathfrak{N}$-Prime Ideals}, University Politehnica of Bucharest Scientific Bulletin-Series A-Applied Mathematics and Physics \textbf{80} (2) (2018), 179-190.

\end{thebibliography}
\end{document}